\documentclass[a4paper,11pt, reqno]{amsart}
\usepackage{amssymb,amsthm,amsmath}
\usepackage{cite}

\usepackage{amsmath,amsthm,amscd,amssymb}
\usepackage{latexsym}
\usepackage[colorlinks,citecolor=red,pagebackref,hypertexnames=false]{hyperref}

 \numberwithin{equation}{section}
\newtheorem{thm}{Theorem}[section]

\newtheorem{lem}[thm]{Lemma}

{\qed\bigskip}

\newtheorem{prop}[thm]{Proposition}

\allowdisplaybreaks[2]

\title[]
{Restriction theorem for the Fourier-Hermite transform and solution of the Hermite-Schr\"odinger equation
}
\author{Shyam Swarup Mondal}

\author{Jitendriya Swain}

\address{Shyam Swarup Mondal  \endgraf Department of Mathematics
	\endgraf IIT Guwahati
	\endgraf Guwahati, Assam, India.}
\email{}

\address{Jitendriya Swain,  \endgraf Department of Mathematics
	\endgraf IIT Guwahati
	\endgraf Guwahati, Assam, India.}
\email{jitumath@iitg.ac.in}

\date{\today}
\begin{document}

	\maketitle
	\begin{abstract}
In this article, we prove the restriction theorem for the Fourier-Hermite transform and obtain the Strichartz estimate  for  the system of orthonormal functions for the Hermite operator $H=-\Delta+|x|^2$ on $\mathbb{R}^n$  as application. Further, we show the an optimal behavior of the constant in the Strichartz estimate as limit of a large number of functions.
\end{abstract}

\section{Introduction}
Let $f\in L^1(\mathbb{R}^n)$. The Fourier-Hermite transform of $f$ is defined by $$\hat{f}(\mu)=\int_{\mathbb{R}^n}f(x)\Phi_\mu(x)\,dx, \quad \mu \in \mathbb{N}_0^n,$$ where $\mathbb{N}_0$ denotes the set of all non-negative integers and $\Phi_\mu$'s are the $n$-dimensional Hermite functions (defined in section 2).
If $f\in L^2(\mathbb{R}^n)$  then $\{\hat{f}(\mu)\}\in \ell^2(\mathbb{N}_0^n)$ and the Plancherel formula is of the form  $$\|f\|_2^2=\displaystyle\sum_{\mu\in{\mathbb{N}^n_0}}|\hat{f}(\mu)|^2.$$ The inverse Fourier-Hermite transform is given by $$f(x)=\sum_{\mu\in\mathbb{N}^n_0}\hat{f}(\mu)\Phi_\mu(x).$$
Given a discrete surface $S$ in $\mathbb{N}_0^n \times \mathbb{Z}$, we define the restriction operator $\mathcal{R}_Sf:=\{\hat{f}(\mu,\nu)\}_{(\mu,\nu)\in S}$  and the operator dual to $\mathcal{R}_S$ (called the extension operator)  as
\begin{align}\label{extension}
\mathcal{E}_S(\{\hat{f}(\mu,\nu)\}):= \sum_{(\mu,\nu)\in S} \hat{f}(\mu,\nu)\Phi_\mu(\cdot) e^{-i(\cdot)\nu}.
\end{align}
 We consider the following problem:

 {\bf Problem 1:} For which exponents $1\leq p\leq 2,$ the sequence of Fourier-Hermite transforms of a function $f\in L^p((-\pi, \pi)\times\mathbb{R}^n)$ belongs to $\ell^2(S)$?

 This question can be reframed to the boundedness of the operator $\mathcal{E}_S$ from $\ell^2(S)$ to $L^{p'}((-\pi, \pi)\times\mathbb{R}^n)$, where $p'$ is the conjugate exponent of $p$. Since $\mathcal{E}_S$ is bounded from $\ell^2(S)$ to $L^{p'}((-\pi, \pi)\times\mathbb{R}^n)$ if and only if $T_S:=\mathcal{E}_S(\mathcal{E}_S)^*$ is bounded from $L^p((-\pi, \pi)\times\mathbb{R}^n)$ to $L^{p'}((-\pi, \pi)\times\mathbb{R}^n)$,
Problem 1 can be re-written as follows:

{\bf Problem 2:} For which exponents $1\leq p\leq 2,$ the operator  $T_S:=\mathcal{E}_S(\mathcal{E}_S)^*$ is bounded from $L^p((-\pi, \pi)\times\mathbb{R}^n)$ to $L^{p'}((-\pi, \pi)\times\mathbb{R}^n)$?

Note that H\"{o}lder's inequality implies that the operator $T_S = \mathcal{E}_S(\mathcal{E}_S)^*$ is bounded from
$L^p((-\pi, \pi)\times\mathbb{R}^n)$ to $L^{p'}((-\pi, \pi)\times\mathbb{R}^n)$ if and only if for any $W_1, W_2\in L^{\frac{2p}{2-p}}((-\pi, \pi)\times\mathbb{R}^n)$, the operator $W_1T_SW_2$ is bounded on $L^2((-\pi, \pi)\times\mathbb{R}^n)$ with
$$\|W_1T_SW_2\|_{L^2((-\pi, \pi)\times\mathbb{R}^n)\to L^2((-\pi, \pi)\times\mathbb{R}^n)}\leq C\|W_1\|_{L^{\frac{2p}{2-p}}((-\pi, \pi)\times\mathbb{R}^n)}\|W_2\|_{L^{\frac{2p}{2-p}}((-\pi, \pi)\times\mathbb{R}^n)},$$ for some $C>0.$

We introduce an analytic family of operators $(T_z)$ defined on the strip $a\leq {\rm Re}~ z\leq b$ in the complex plane  such that $T_S=T_c$ for some $c\in (a, b)$ and show that the operator $W_1T_SW_2$ belongs to a Schatten class with $$ \|W_1T_SW_2\|_{\mathcal{G}^\alpha(L^2((-\pi, \pi)\times\mathbb{R}^n))}\leq C\|W_1\|_{L^{\frac{2p}{2-p}}((-\pi, \pi)\times\mathbb{R}^n)}\|W_2\|_{L^{\frac{2p}{2-p}}((-\pi, \pi)\times\mathbb{R}^n)},$$ for some $C>0$ and some $\alpha>0.$
 which is more general  result $L^p-L^{p'}$ boundedness of  $T_S.$

Such problems are often considered in the literature. For example, the celebrated Stein-Tomas Theorem (see \cite{Stein,T1,T2}) gives an affirmative answer to Fourier restriction problem for compact surfaces with non-zero Gaussian curvature if and only if $1\leq p\leq \frac{2(n+1)}{n+3}$.  For quadratic surfaces, Strichartz \cite{st}  gave a complete solution to Fourier restriction problem, when $S$ is a quadratic surface given by $S = \{x \in \mathbb{R}^n: R(x) = r\},$ where $R(x)$ is a polynomial of
degree two with real coefficients and $r$ is a real constant. Further the Stein-Tomas Theorem is generalized to a system of orthonormal functions with respect to the Fourier transform by Frank-Lewin-Lieb-Seiringer \cite{frank1} and Frank-Sabin \cite{frank}.

The main aim of the paper is to prove the restriction theorem for the system of orthonormal functions with respect to the Fourier-Hermite transform and discuss some of its applications to PDEs. Although the Strichartz inequality for  the system of orthonormal functions for the Hermite operator has been  proved in \cite{lee} using the classical Strichartz estimates for the free Schr\"odinger propagator for  orthonormal systems \cite{frank, frank1} and the link between the Schr\"odinger kernel and the Mehler kernel associated with  the Hermite semigroup \cite{SjT}. It is important to note that this result can also be obtained independently as a direct application of the Fourier-Hermite restriction theorem. To the best of our knowledge, the study on restriction theorem with respect to the Fourier-Hermite transform has not been considered in the literature so far. However, we prove the restriction theorem for the Fourier-Hermite transform and obtain the full range Strichartz estimate for the system of orthonormal functions for the Hermite operator $H=-\Delta+|x|^2$ on $\mathbb{R}^n$  as an application.
 we also show that the constant obtained the Strichartz inequality is optimal in terms of the limit of a large number of functions. Also, we discuss the global well-posedness results in Schatten spaces for the nonlinear Hermite-Hartree equation as an application to our main result.

The schema of the paper   is as follows: In Section \ref{CH2sec2},  we  discuss the spectral theory of the  Hermite operator and the kernel estimates for the Hermite semigroup. 
In Section \ref{CH2sec3}, we obtain the duality principle in terms of Schatten bounds of the operator $We^{-itH}(e^{-itH})^*\overline{W}$ and give an affirmative answer to Problem 2 when $p =\frac{2\lambda_0}{1+\lambda_0}$ for some $\lambda_0>1.$
In Section \ref{CH2sec4}, we obtain the Strichartz estimate for $1\leq q<\frac{n+1}{n-1}$, for the system of  orthonormal functions associated with  the Hermite operator as the restriction of the Hermite-Fourier transform to the discrete surface $S=\{(\mu, \nu)\in \mathbb{N}_0^n\times \mathbb{Z}: \nu=2|\mu|+n\}$. In Sections \ref{CH2sec5}, we prove the optimality of Schatten exponent and we obtain the global  well-posedness result for the non-linear  Hermite-Hartree equation in Schatten spaces.

\section{Preliminary}\label{CH2sec2}
In this section we discuss some basic definitions and provide necessary background information  about the Hermite semigroup.
\subsection{Hermite Operator and the Spectral theory}
Let $\mathbb{N}_0$ be the set of all non-negative integers. Let $H_k$ denote the Hermite polynomial on $\mathbb{R}$, defined by
$$H_k(x)=(-1)^k \frac{d^k}{dx^k}(e^{-x^2} )e^{x^2}, \quad k\in \mathbb{N}_0$$
and $h_k$ denote the normalized Hermite functions on $\mathbb{R}$ defined by
$$h_k(x)=(2^k\sqrt{\pi} k!)^{-\frac{1}{2}} H_k(x)e^{-\frac{1}{2}x^2}, \quad k\in \mathbb{N}_0.$$ 
The higher dimensional Hermite functions denoted by $\Phi_{\alpha}$ are obtained by taking tensor product of one dimensional Hermite functions. Thus for any multi-index $\alpha \in \mathbb{N}_0^n$ and $x \in \mathbb{R}^n$, we define
 $\Phi_{\alpha}(x)=\prod_{j=1}^{n}h_{\alpha_j}(x_j).$
The family $\{\Phi_{\alpha}\}$ forms an orthonormal basis for $L^2(\mathbb{R}^n)$. They are eigenfunctions of the Hermite operator $H=-\Delta+|x|^2$ corresponding to eigenvalues $(2|\alpha|+n)$, where $|\alpha |=\sum_{j=1}^{n}\alpha_j$. Given $f \in L^{2}(\mathbb{R}^{n})$, we have the Hermite expansion
$$f=\sum_{\alpha \in \mathbb{N}_0^{n}}\left(f, \Phi_{\alpha}\right) \Phi_{\alpha}=\sum_{k=0}^\infty \sum_{|\alpha|=k}\left(f, \Phi_{\alpha}\right) \Phi_{\alpha}= \sum_{k=0}^\infty P_kf,$$ where $P_{k}$ denotes the orthogonal projection of $L^{2}(\mathbb{R}^{n})$ onto the eigenspace spanned by $\{\Phi_{\alpha} :|\alpha|=k\}.$
The operator $H$ defines a semigroup called the Hermite semigroup
$e^{-t H}, t>0,$   defined by
$$
e^{-t H} f=\sum_{k=0}^{\infty} e^{-(2 k+n) t} P_{k} f
$$
for $f \in L^{2}(\mathbb{R}^{n}).$ On a dense subspace, say the space of all Schwartz functions, the above can be written as
$$
e^{-t H} f(x)=\int_{\mathbb{R}^{n}} f(y) K_{t}(x, y) d y,
$$ where the kernel $K_{t}(x, y)$ is given by the expansion
$$
K_{t}(x, y)=\sum_{\alpha \in \mathbb{N}^{n}_0} e^{-(2|\alpha|+n) t} \Phi_{\alpha}(x) \Phi_{\alpha}(y).
$$


For $z'=r+i t, r>0, t \in \mathbb{R},$ the kernel of the operator $e^{-z'H}$ is given by
$$
K_{z'}(x, y)=\sum_{k=0}^{\infty} e^{-z'(2 k+n)} \sum_{|\alpha|=k}\Phi_{\alpha}(x) \Phi_{\alpha}(y).
$$
Using Mehler's formula, the kernel of the operator $e^{-z'H}$ can be obtained as
\begin{align*}
K_{z'}(x, y)&=\frac{1}{(2\pi\sinh 2z')^{\frac{n}{2}}} e^{\frac{1}{2}\left( -\coth 2z'(|x|^2+|y|^2)+\frac{2 x\cdot y}{\sinh  2z'}\right) }.
\end{align*}
For $t\in \mathbb{R}\setminus(\frac{\pi}{2})\mathbb{Z}$, letting $r\to 0$, the kernel of the operator $e^{-itH}$ can be written as
\begin{align}\label{CH2k}
K_{it}(x, y)&=\frac{e^{-\frac{i\pi n}{4}}}{(2\pi\sin 2t)^{\frac{n}{2}}} e^{\frac{i}{2}\left( \cot 2t(|x|^2+|y|^2)-\frac{2 x\cdot y}{\sin  2t}\right) }.
\end{align} For $t\in \mathbb{R}\setminus(\frac{\pi}{2})\mathbb{Z}$, \begin{eqnarray}\label{CH2p} K_{-it}(x, y)=\overline{K_{it}(x, y)}\quad \mbox{and}\quad K_{i(t+\frac{\pi}{2})}(x, y)=e^{-i\pi\frac{n}{2}}K_{it}(-x, y).\end{eqnarray} For real valued functions $f$ the $L^p(\mathbb{R}^n)$ norm of $e^{-itH}f$ is even and $\frac{\pi}{2}$-periodic as a function of $t.$

We refer to \cite{Ratna,Thanga} for a detailed study on the kernel associated with the operator $e^{-itH}$.
\subsection{Schatten class and the duality principle} Let $\mathcal{H}$ be a complex and separable Hilbert space in which the inner product is denoted by $\langle, \rangle_\mathcal{H}$. Let $T:\mathcal{H} \rightarrow \mathcal{H}$ be a compact operator and let    $T^{*}$ denotes the adjoint of $T$.  For $1 \leq r<\infty,$ the Schatten space $\mathcal{G}^{r}(\mathcal{H})$ is defined as the space of all compact operators $T$ on $\mathcal{H}$ such that $$\sum_{n =1}^\infty  \left(s_{n}(T)\right)^{r}<\infty,$$ where $s_{n}(T)$ denotes the singular values of \(T,\) i.e., the eigenvalues of \(|T|=\sqrt{T^{*} T}\) counted according to multiplicity. For $T\in \mathcal{G}^{r}(\mathcal{H})$, the Schatten $r$-norm is defined by $$ \|T\|_{\mathcal{G}^{r}}=\left(\sum_{n=1}^{\infty}\left(s_{n}(T)\right)^{r}\right)^{\frac{1}{r}}. $$An operator belongs to the class \(\mathcal{G}^{1}({\mathcal{H}})\) is known as {\it Trace class} operator. Also, an operator belongs to   \(\mathcal{G}^{2}({\mathcal{H}})\) is known as  {\it Hilbert-Schmidt} operator.
  \section{The Restriction Theorem}\label{CH2sec3}
 In this section, we set a platform to prove the restriction theorem with respect to the Fourier-Hermite transform for a given discrete surface $S\subset \mathbb{N}_0^n\times\mathbb{Z}.$
 The following proposition assures an affirmative answer to Problem 2 under certain assumptions.

In order to obtain the Strichartz inequality for the system of orthonormal functions we need the duality principle lemma in our context. We refer to Proposition 1 and Lemma 3 of \cite{frank1} with appropriate modifications to obtain the following two results:

\begin{prop}\label{CH2prop1}
	Let $(T_z)$ be an analytic family of operators on $(-\pi, \pi)\times \mathbb{R}^n $ in the sense of Stein defined on the strip $-\lambda_0\leq  \operatorname{Re} z\leq 0$
	for some $\lambda_0 > 1$. Assume that we have the following bounds
	\begin{eqnarray}\label{CH2expo}\nonumber \left\|T_{i s}\right\|_{L^{2}((-\pi, \pi)\times \mathbb{R}^n ) \rightarrow L^{2}((-\pi, \pi)\times \mathbb{R}^n )} \leq M_{0} e^{a|s|},~\left\|T_{-\lambda_{0}+i s}\right\|_{L^{1}((-\pi, \pi)\times \mathbb{R}^n ) \rightarrow L^{\infty}((-\pi, \pi)\times \mathbb{R}^n )} \leq M_{1} e^{b|s|}\\ \end{eqnarray}
	for all $ s \in \mathbb{R}$, for some $a, b ,M_{0}, M_{1} \geq 0 $. Then, for all  $W_{1}, W_{2} \in L^{2 \lambda_{0}}\left((-\pi, \pi)\times \mathbb{R}^n , \mathbb{C}\right)$
	the operator $W_{1} T_{-1} W_{2}$ belongs to $\mathcal{G}^{2 \lambda_{0}}\left(L^{2}\left((-\pi, \pi)\times \mathbb{R}^n \right)\right)$ and we have the estimate
	\begin{align}\label{CH2510}
	\left\|W_{1} T_{-1} W_{2}\right\|_{\mathcal{G}^{2 \lambda_{0}}\left(L^{2}\left((-\pi, \pi)\times \mathbb{R}^n \right)\right)} \leq M_{0}^{1-\frac{1}{\lambda_{0}}} M_{1}^{\frac{1}{\lambda_{0}}}\left\|W_{1}\right\|_{L_t^{2 \lambda_{0}}L_x^{2 \lambda_{0}}\left((-\pi, \pi)\times \mathbb{R}^n \right)}\left\|W_{2 }\right\|_{L_t^{2 \lambda_{0}}L_x^{2 \lambda_{0}}\left((-\pi, \pi)\times \mathbb{R}^n \right)}.
	\end{align}
\end{prop}
%
%
%
 \begin{lem}\label{CH2dual}(Duality principle)
 	Let $p,q\geq 1$ and $\alpha\geq1$. Let $A$ be a bounded linear operator from  $ L^2(\mathbb{R}^n)$ to ${L_t^{\frac{q'}{2}}L_x^{\frac{p'}{2}}((-\pi, \pi)\times \mathbb{R}^{n})}.$  Then the following statements are equivalent.
 	\begin{enumerate}
 		\item There is a constant \(C>0\) such that
 		\begin{align}\label{CH2511}
 		\left\|W A A^{*} \overline{W}\right\|_{\mathcal{G}^{\alpha}\left(L^{2}\left( (-\pi, \pi) \times\mathbb{R}^{n} \right)\right)} \leq C\|W\|_{{L_t^{\frac{2q}{2-q}}L_x^{\frac{2p}{2-p}}((-\pi, \pi)\times \mathbb{R}^{n})}}^{2} \end{align} for all $W \in {L_t^{\frac{2q}{2-q}}L_x^{\frac{2p}{2-p}}((-\pi, \pi)\times \mathbb{R}^{n})}$, where the function $W$ is interpreted as an operator which acts by multiplication.
 		
 		\item  For any orthonormal system $\left(f_{j}\right)_{j \in J}$
 		in  $L^2(\mathbb{R}^n)$ and any sequence $\left(n_{j}\right)_{j \in J} \subset \mathbb{C}$, there is a constant \(C'>0\) such that
 		\begin{align}\label{CH2512}
 		\left\|\sum_{j \in J} n_{j} \left| A f_{j}\right|^{2}\right\|_{L_t^{\frac{q'}{2}}L_x^{\frac{p'}{2}}((-\pi, \pi)\times \mathbb{R}^{n})} \leq C' \left(\sum_{j \in J}\left|n_{j}\right|^{\alpha^{\prime}}\right)^{1 / \alpha^{\prime}}.
 		\end{align}
 		 	\end{enumerate}
\end{lem}
Note that Lemma \ref{CH2dual} and Proposition \ref{CH2prop1} are also valid in the domain $(-\frac{\pi}{4}, \frac{\pi}{4})\times \mathbb{R}^{n}.$

 Let $S$ be the discrete surface $S=\{ (\mu, \nu) \in \mathbb{N}_0^{n}\times\mathbb{Z}: R(\mu, \nu)=0\},$ where $R(\mu, \nu)$ is a polynomial of degree one,  with respect to the counting measure.

 	For $-1< {\rm Re}~ z\leq 0$, consider the analytic family of generalized functions
 	\begin{align}\label{CH2gf}
 	G_z(\mu , \nu)=\psi(z)R(\mu, \nu)_+^z,
 	\end{align}
 	where $\psi(z)$ is an appropriate analytic function with a simple zero at $z=-1$ with  exponential growth at infinity when $\operatorname{Re}(z)=0 $ and
 	$$R(\mu, \nu)_+^z=\begin{cases}
 	R(\mu, \nu)^z~\text{ for }  R(\mu, \nu)>0,\\0~\quad\quad\quad \text{ for }R(\mu, \nu)\leq0.
 	\end{cases} $$
 	For Schwartz class functions $\phi$ on $\mathbb{N}_0^{n}\times\mathbb{Z},$ we have $	\left\langle G_{z}, \phi\right\rangle:=\psi(z)\sum_{\mu, \nu}R(\mu, \nu)_+^z  \phi(\mu, \nu) $ and
 	$\displaystyle\lim_{z\to -1}\left\langle G_{z}, \phi\right\rangle=\lim_{z\to -1}\psi(z) \sum_{\mu,\nu}\phi(\mu, \nu)R(\mu, \nu)_+^z
 	=\sum_{(\mu,\nu)\in S}\phi(\mu, \nu).$ We refer to \cite{sh} for the distributional calculus of $R(\mu, \nu)_+^z.$
 	Thus $G_{-1}=\delta_S.$ For $-1< {\rm Re}~ z\leq 0$, define the analytic family of operators $T_z$ (on Schwartz class functions on $(-\pi, \pi)\times \mathbb{R}^n$) by
 	$$T_z g(t, x) = \displaystyle\sum_{\mu, \nu} \hat{g}(\mu, \nu) G_z( \mu, \nu) \Phi_\mu(x) e^{-i\nu t },$$ where $\displaystyle\hat{g}(\mu, \nu)=\int_{\mathbb{R}^n}\int_{(-\pi, \pi)}g(t, x)\Phi_\mu(x)e^{i\nu t}\,dx dt.$ Then $(T_z)$  is an analytic in the sense of Stein defined on the strip $-\lambda_0\leq  \operatorname{Re} z\leq 0$ for some $\lambda_0>1$ and
 	\begin{align}\label{CH222-22}T_z g(t, x) 
 	&= \int_{\mathbb{R}^n }  (K_z(x, y, \cdot )*g(y, \cdot))(t)   ~dy,
 	\end{align}
 	where $ K_z(x, y, t) =\displaystyle \sum_{\mu, \nu}   \Phi_\mu(x) \Phi_\mu(y)   G_z( \mu, \nu) e^{-i\nu t }.$  When $Re(z)=0$, we have
 	\begin{align}\label{CH2twoo}
 	\|T_{is} \|_{L^2((-\pi, \pi)\times \mathbb{R}^n)\to L^2((-\pi, \pi)\times \mathbb{R}^n)} =\left\|G_{i s}\right\|_{L^{\infty}\left((-\pi, \pi)\times \mathbb{R}^n \right)}\leq \left|  \psi(is)\right|.	\end{align}
 	
 	Again an application of H\"older and Young inequalities in (\ref{CH222-22}) gives \begin{eqnarray}\label{CH21inf}
 	|T_z g(t, x) |\leq  \sup_{t\in (-\pi, \pi),~ y\in \mathbb{R}^n}|K_z(x, y, t )| \|g\|_{L^1((-\pi, \pi)\times\mathbb{R}^n)}
 	\end{eqnarray} for $g\in L^1((-\pi, \pi)\times\mathbb{R}^n )$. Denoting $T_S=\mathcal{E}_S\mathcal{E}_S^*$, we obtain the following Schatten bound (see (\ref{CH2SB}) below) for the operator $W_{1} T_{S} W_{2}$, where $W_1,W_2\in L_{t, x}^{2\lambda_0}\left((-\pi, \pi)\times \mathbb{R}^n\right)$.
 	
 	  	\begin{thm}\label{CH2diagggg} (Fourier-Hermite restriction theorem)
 	Let $n \geq 1$ and let $S \subset \mathbb{N}_0^{n}\times\mathbb{Z}$ be a discrete  surface. Suppose that  for each $x,y\in \mathbb{R}$,    $|K_z(x,y,t)|$ is bounded  and has at most exponential growth at infinity when  $z=-\lambda_{0}+is$ for some $\lambda_0>1$. Then $T_S$ is bounded from $L^p((-\pi, \pi)\times\mathbb{R}^n)$ to $L^{p'}((-\pi, \pi)\times\mathbb{R}^n)$ for $p=\frac{2\lambda_0}{1+\lambda_0}$.
 	
 \end{thm}
\begin{proof} It is enough to show \begin{eqnarray}\label{CH2SB}\left\|W_{1} T_{S} W_{2}\right\|_{\mathcal{G}^{2\lambda_0}\left(L^{2}\left((-\pi, \pi)\times \mathbb{R}^n\right)\right)}  \leq C\left\|W_{1}\right\|_{L_{t, x}^{2\lambda_0}\left((-\pi, \pi)\times \mathbb{R}^n\right)} \left\|W_{2}\right\|_{L_{t, x}^{2\lambda_0}\left((-\pi, \pi)\times \mathbb{R}^n\right)}\end{eqnarray} for all $W_1, W_2\in L_{t, x}^{2\lambda_0}\left((-\pi, \pi)\times \mathbb{R}^n\right)$ by Lemma \ref{CH2dual}. By our assumption, together with (\ref{CH2twoo}), (\ref{CH21inf}) and Proposition \ref{CH2prop1}, we get (\ref{CH2SB}).
\end{proof}
 \section{Strichartz inequality for system of  orthonormal functions}\label{CH2sec4}
 Consider the Schr\"odinger equation associated with the Hermite operator $H=-\Delta+|x|^2$:
\begin{align}\label{CH2503}
i \partial_{t} u(t, x)&=H u(t, x) \quad t \in \mathbb{R},~x \in \mathbb{R}^{n} \\\nonumber u(x, 0)&=f(x).
 \end{align}
If  $f\in L^2(\mathbb{R}^n)$, the solution of the initial value problem (\ref{CH2503}) is given  by $u(t,x)=e^{-i t H} f(x).$ The solution to the initial value problem (\ref{CH2503}) can be realized as the extension operator of some function $f$ on $(-\pi, \pi)\times\mathbb{R}^n.$ To estimate the solution to the initial value problem (\ref{CH2503}) is equivalent to obtain the Schatten bound (\ref{CH2511}) with $A=e^{-i t H}$.

Let $S$ be the discrete surface $S=\{(\mu, \nu)\in \mathbb{N}_0^n\times \mathbb{Z}: \nu=2|\mu|+n\}$ with respect to the counting measure. Then for all $\hat{f} \in \ell^{1}(S)$ and for all $(t, x)\in [-\pi,\pi]\times \mathbb{R}^n$, the extension operator can be written as
\begin{align}\label{CH2surface}
\mathcal{E}_{S} f(t, x)=\sum_{\mu, \nu\in S}\hat{f}(\mu, \nu) \Phi_\mu(x)e^{-it\nu},
\end{align}
where $\displaystyle\hat{f}(\mu, \nu)=\int_{\mathbb{R}^n}\int_{(-\pi, \pi)}f(t, x)\Phi_\mu(x)e^{i\nu t}\,dx dt.$ Choosing
$$
\hat{f}(\mu, \nu)=\left\{\begin{array}{ll}       {\hat{u}(\mu)} & {\text { if } \nu=2|\mu|+n,}\\{0} & {\text {otherwise,} }\end{array} \right.
$$ for some $u:\mathbb{R}^n\to \mathbb{C}$
in (\ref{CH2surface}), we get
\begin{align*}
\mathcal{E}_{S} f(t, x)&=\sum_{\mu, \nu\in S}\hat{f}(\mu, \nu) \Phi_\mu(x)e^{-it\nu}\\
&=\int_{\mathbb{R}^n} \left( \sum_{\mu} \Phi_\mu(x)\Phi_\mu(y) e^{-it (2|\mu|+n)}\right) u(y)~dy\\
&=e^{-itH}u(x).
\end{align*}
Setting $\psi(z)=\frac{1}{\Gamma(z+1)}$ and $R(\mu,\nu)=\nu-(2|\mu|+n)$ in (\ref{CH2gf}), we get
\begin{align*}
G_z(\mu , \nu)=\frac{1}{\Gamma(z+1)}(\nu-(2|\mu|+n))_+^z
\end{align*} and for Schwartz class functions $\Phi$ on $\mathbb{N}_0^{n}\times\mathbb{Z},$
	$$\displaystyle\lim_{z\to -1}\left\langle G_{z}, \phi\right\rangle=\lim_{z\to -1}\frac{1}{\Gamma(z+1)} \sum_{\mu,\nu}\phi(\mu, \nu)(\nu-(2|\mu|+n))_+^z=\sum_{(\mu,\nu)\in S}\phi(\mu, \nu).$$
		Thus $G_{-1}=\delta_S$ and
\begin{align}\label{CH22-2}T_z g(t, x) 
&= \int_{\mathbb{R}^n }  (K_z(x, y, \cdot )*g(y, \cdot))(t)   ~dy,
\end{align}
where \begin{align}\label{CH2kernl}
K_z(x, y, t )&
=\frac{1}{\Gamma(z+1)}\sum_{\mu}   \Phi_\mu(x) \Phi_\mu(y)   e^{-i t(2|\mu|+n ) }  \sum_{k=0 }^\infty k_+^z e^{-i tk }.
\end{align}
We first calculate $ \displaystyle\sum_{k=0 }^\infty k_+^z e^{-i tk }$ in the following proposition.


\begin{prop}\label{CH2515}
Let $-1< {\rm Re}~ z<0$. Then the series
$ \sum_{k=0 }^\infty k_+^z e^{-i tk }$ is the Fourier series of a integrable function on $[-\pi, \pi]$ which is of class $C^\infty $ on $[-\pi, \pi]\setminus  \{0\}.$ Near  origin this function has the same singularity as the function whose values are $ \Gamma(z+1)(it)^{-z-1}$, i.e.,
\begin{align}\label{CH2eq4} \sum_{k=0 }^\infty k_+^z e^{-i tk }\sim  \; \Gamma(z+1)(it)^{-z-1}+b(t),\end{align}
where $b\in C^\infty[-\pi , \pi]$.
 \end{prop}
\begin{proof}
 For $\tau>0$,  we calculate the inverse Fourier  transform of  $u_+^z e^{-\tau u }.$
  \begin{align*}
  \mathcal{F}^{-1}[ u_+^z e^{-\tau u }](x)&=\int_\mathbb{R}  u_+^z e^{-\tau u } e^{-i ux }\;du=\int_0^\infty   u^z e^{-i s u}\;du,
  \end{align*}
    where $s=x-i\tau$ so that $-\pi<\arg s<0$. Then   $ u_+^z e^{-\tau u }$ converges to  $u_+^z  $  in the sense of distributions as $\tau \to 0$. Also,  the inverse Fourier  transform of  $u_+^z e^{-\tau u }$ converges to  the inverse Fourier  transform of  $u_+^z  $. Using the change of variable $isu=\xi$ and proceeding as in page 170 of \cite{sh}, we get

    \begin{align}\label{CH2tau}
  \mathcal{F}^{-1}[ u_+^z e^{-\tau u }](x)& =\frac{1}{(is)^{z+1}}\int_L  \xi^z e^{-\xi} \;d\xi=\frac{\Gamma(z+1)}{(is)^{z+1}},
\end{align}
  where   the contour $L$ of the integral is a ray from origin to infinity whose angle with respect to the real axis is given by $\arg\xi=\arg s+\frac{\pi}{2}$.
 Letting  $\tau \to 0$ in (\ref{CH2tau}),  we have
   \begin{align}\label{CH2eq50}
  \mathcal{F}^{-1}[ u_+^z ](x)& =\frac{\Gamma(z+1)}{(ix)^{z+1}}.
  \end{align}
  By analytic continuation (\ref{CH2eq50}) is valid for all $z\neq -1$.

We use the idea given in Theorem 2.17 of \cite{stein11} to prove (\ref{CH2eq4}). To make the paper self contained, we will only indicate the main steps. Let us consider a function $\eta\in C^\infty(\mathbb{R})$ such that $\eta(x)=1$ if $|x|\geq 1$, and vanishes in a neighborhood of the origin. Let $F(x)=\eta(x)x_{+}^z $ for $x\in  \mathbb{R}$. Writing  $F(x)=x_{+}^z +(\eta(x)-1)x_{+}^z $, using (\ref{CH2eq50}) and denoting $f$ to be the inverse Fourier transform of $F$ in the sense of distributions, we have  $f(x)= {\Gamma(z+1)}{(ix)^{-z-1}}+b_1(x)$, where $b_1$ is the inverse Fourier transform of the   integrable function 
$(\eta(x)-1)x_{+}^z$  whose support is bounded.  Moreover, $b_1\in C^\infty(\mathbb{R})$ and $f\in L^1(\mathbb{R}).$

Applying Poisson summation formula (see page  250 of  \cite{stein11}) to the function $f$ and using the fact   $\hat{f}=F$, we get
\begin{align*}
  \sum_{k=0 }^\infty k_+^z e^{-i tk }&= \sum_{k \in \mathbb{Z}} F(k)e^{-i tk }\\
  &\sim  \sum_{k\in \mathbb{Z}} f(2k\pi+t)  \\
  &=f(t)+ \sum_{|k|>0} f(2k\pi+t)  \\
  &= {\Gamma(z+1)}{(it)^{-z-1}}+b_1(t)+ \sum_{|k|>0} f(2k\pi+t)\\
  &= \Gamma(z+1)(it)^{-z-1}+b(t),
\end{align*}
where $b(t)=b_1(t)+ \sum_{|k|>0} f(2k\pi+t)\in C^\infty[-\pi , \pi]$.

\end{proof}

Now we are in a position to prove the following Strichartz inequality for the diagonal case.
\begin{thm}\label{CH2diag} (Diagonal case)
	Let 	$n \geq 1$. Then
for any (possibly infinite) system $\left(u_{j}\right)$ of orthonormal functions in $L^{2}\left(\mathbb{R}^{n}\right)$
	and any coefficients $ \left( n_{j} \right) \subset \mathbb{C},$ we have
	\begin{align}\label{CH2d}
	\left\| \sum_{j} n_{j}\left| e^{-i t H} u_{j}\right|^2\right\|_{L_{t,x}^{1+\frac{2}{n}}((-\pi, \pi)\times\mathbb{R}^n)}  \leq C\left(\sum_{j}\left|n_{j}\right|^{\frac{n+2 }{n+1}}\right)^{\frac{n+1}{n+2}},
	\end{align}
	where $C>0$ is independent of on $n$ and $q$.

\end{thm}
\begin{proof}  To prove (\ref{CH2d}), it is enough to show \begin{eqnarray}\label{CH2period}\left\|W_{1} T_{S} W_{2}\right\|_{\mathcal{G}^{n+2}\left(L^{2}\left((-\frac{\pi}{4}, \frac{\pi}{4})\times \mathbb{R}^n\right)\right)}  \leq C\left\|W_{1}\right\|_{L_{t, x}^{n+2}\left((-\frac{\pi}{4}, \frac{\pi}{4})\times \mathbb{R}^n\right)} \left\|W_{2}\right\|_{L_{t, x}^{n+2}\left((-\frac{\pi}{4}, \frac{\pi}{4})\times \mathbb{R}^n\right)}\end{eqnarray} for all $W_1, W_2\in L_{t, x}^{n+2}\left((-\frac{\pi}{4}, \frac{\pi}{4})\times \mathbb{R}^n\right),$ where $S=\{(\mu, \nu)\in \mathbb{N}_0^n\times \mathbb{Z}: \nu=2|\mu|+n\}$. Since applying Lemma \ref{CH2dual} to (\ref{CH2period}) gives \begin{eqnarray}\label{CH2p1}\left\| \sum_{j} n_{j}\left| e^{-i t H} u_{j}\right|^2\right\|_{L_{t,x}^{1+\frac{2}{n}}((-\frac{\pi}{4}, \frac{\pi}{4})\times\mathbb{R}^n)}  \leq C\left(\sum_{j}\left|n_{j}\right|^{\frac{n+2 }{n+1}}\right)^{\frac{n+1}{n+2}}.\end{eqnarray} Using the kernel properties (\ref{CH2p}) of the semigroup $e^{-itH}$ the range of $t$ can be extended to $(-\pi,\pi)$.
We show that the family of operators $(T_z)$ (defined in (\ref{CH22-2})) satisfies (\ref{CH2expo}).  	When $Re(z)=0$, we have
	\begin{align}\label{CH2two}
	\|T_{is} \|_{L^2((-\frac{\pi}{4}, \frac{\pi}{4})\times \mathbb{R}^n)\to L^2((-\frac{\pi}{4}, \frac{\pi}{4})\times \mathbb{R}^n)} =\left\|G_{i s}\right\|_{L^{\infty}\left((-\frac{\pi}{4}, \frac{\pi}{4})\times \mathbb{R}^n \right)}\leq \left|  \frac{1}{\Gamma(1+i s)} \right|\leq C e^{\pi|s| / 2}.	\end{align}
When $z=-\lambda_{0}+is$, (\ref{CH21inf}) gives $T_z$  is bounded from $ L^1((-\frac{\pi}{4}, \frac{\pi}{4})\times \mathbb{R}^n)$ to $ L^\infty ((-\frac{\pi}{4}, \frac{\pi}{4})\times \mathbb{R}^n)$ if and only if $|K_z(x,y,t)|$ is bounded for each $x, y\in\mathbb{R}^n.$
But by (\ref{CH2kernl}), Proposition \ref{CH2515} and (\ref{CH2k}), we get
\begin{align}\label{CH2kernel}
|K_z(x, y, t)|&\sim\frac{C}{ |t|^{{\rm Re}~(z+1+\frac{n}{2})}}e^{iz\frac{\pi}{2}}.
\end{align}
So for each $x,y\in \mathbb{R}$, $|K_z(x,y,t)|$ is bounded if and only if ${\rm Re}~(z)=-\frac{n+2}{2}.$   The conclusion of the theorem follows by choosing $\lambda_0=\frac{n+2}{2}$ by Proposition \ref{CH2prop1} and the identity $T_S=T_{-1}.$
\end{proof}

To obtain the Strichartz inequality for the general case we need to observe the following.
\begin{thm}\label{CH2fr}
	Let $S$ be the discrete surface $S=\{(\mu, \nu)\in \mathbb{N}_0^n\times \mathbb{Z}: \nu=2|\mu|+n\}$ with respect to the counting measure.
	Then for all exponents $p, q \geq 1$ satisfying
	$$
	\frac{2}{p}+\frac{n}{q}=1, \quad q>n+1
	$$
	we have
	\begin{eqnarray}\label{CH2p2}
	\left\|W_{1} T_{S} W_{2}\right\|_{\mathcal{G}^{q}\left(L^{2}\left((-\pi, \pi)\times \mathbb{R}^{n}\right)\right)} \leq C\left\|W_{1}\right\|_{L_{t}^{p} L_{x}^{q}\left((-\pi, \pi)\times \mathbb{R}^{n}\right)}\left\|W_{2}\right\|_{L_{t}^{p} L_{x}^{q}\left((-\pi, \pi)\times \mathbb{R}^{n}\right)}
	\end{eqnarray}
	with $C>0$ independent of $W_{1}, W_{2}$.
\end{thm}
\begin{proof}
  The operator $ T_{-\lambda_{0}+i s}  $ is  an integral operator with kernel $K_{{-\lambda_{0}+is}}(x, x', t-t')$ defined in (\ref{CH22-2}). An application of  Hardy-Littlewood-Sobolev inequality along with (\ref{CH2two}) and (\ref{CH2kernel}) yields
\begin{align*}
&	\left\|W_{1}^{\lambda_{0}-i s} T_{-\lambda_{0}+i s} W_{2}^{\lambda_{0}-i s}\right\|_{\mathcal{G}^{2}}^{2}\\
&=\int_{(-\frac{\pi}{4}, \frac{\pi}{4})^{2}}\int_{\mathbb{R}^{2n}} W_{1}(t, x)^{2 \lambda_{0}}\left|K_{-\lambda_{0}+i s}(x,t,  x', t')\right|^{2} W_{2}\left(t^{\prime}, x^{\prime}\right)^{2 \lambda_{0}} d x d x^{\prime} d t d t^{\prime}\\
&\leq C_1\int_{(-\frac{\pi}{4}, \frac{\pi}{4})^{2}}\int_{\mathbb{R}^{2n}} \frac{W_{1}(t, x)^{2 \lambda_{0}} W_{2}\left(t^{\prime}, x^{\prime}\right)^{2 \lambda_{0}} }{\left|t-t^{\prime}\right|^{n+2-2 \lambda_{0}}}d x d x^{\prime} d t d t^{\prime}\\
	&\leq C_1e^{\pi |s|}    \int_{(-\frac{\pi}{4}, \frac{\pi}{4})} \int_{(-\frac{\pi}{4}, \frac{\pi}{4})} \frac{\left\|W_{1}(t)\right\|_{L_{x}^{2 \lambda}\left(\mathbb{R}^{n}\right)}^{2 \lambda_{0}}\left\|W_{2}\left(t^{\prime}\right)\right\|_{L_{x}^{2}}^{2 \lambda_{0}}\left(\mathbb{R}^{n}\right)}{\left|t-t^{\prime}\right|^{n+2-2 \lambda_{0}}} d t d t^{\prime}\\
	&\leq C_1 e^{\pi |s|}  \|W_{1}\|_{ L_{t}^{\frac{4 \lambda_{0}}{2 \lambda_{0}-n}} L_{x}^{2 \lambda_{0}}\left((-\frac{\pi}{4}, \frac{\pi}{4}) \times \mathbb{R}^{n}\right)}^{2 \lambda_{0}}
	\|W_{1}\|_{ L_{t}^{\frac{4 \lambda_{0}}{2 \lambda_{0}-n}} L_{x}^{2 \lambda_{0}}\left((-\frac{\pi}{4}, \frac{\pi}{4}) \times \mathbb{R}^{n}\right)}^{2 \lambda_{0}}
	\end{align*}
provided we have  $0 \leq n+2-2 \lambda_{0}<1,$ that is $\frac{n+1}{2}<\lambda_{0} \leq \frac{n+2}{2} .$   By  Theorem 2.9  of  \cite{siman}, we have
	$$
	\left\|W_{1} T_{-1} W_{2}\right\|_{\mathcal{G}^{2 \lambda_{0}}\left(L^{2}((-\frac{\pi}{4}, \frac{\pi}{4}) \times\mathbb{R}^{n})\right)} \leq C\|W_{1}\|_{ L_{t}^{\frac{4 \lambda_{0}}{2 \lambda_{0}-n}} L_{x}^{2 \lambda_{0}}\left((-\frac{\pi}{4}, \frac{\pi}{4}) \times \mathbb{R}^{n}\right)}
	\|W_{1}\|_{ L_{t}^{\frac{4 \lambda_{0}}{2 \lambda_{0}-n}} L_{x}^{2 \lambda_{0}}\left((-\frac{\pi}{4}, \frac{\pi}{4}) \times \mathbb{R}^{n}\right)}.
	$$
By Lemma \ref{CH2dual} and (\ref{CH2p}) (as in the proof of Theorem \ref{CH2diag}) we get (\ref{CH2p2}).
\end{proof}
\begin{thm}\label{CH2STH}(Strichartz inequality for orthonormal functions for Hermite operator)
	Let
	$p, q, n \geq1$ such that  $$ 1\leq q < \frac{n+1}{n-1}\quad \text { and } \quad \frac{2}{p}+\frac{n}{q}=n.$$ For any (possibly infinite) system $\left(u_{j}\right)$ of orthonormal functions in $L^{2}\left(\mathbb{R}^{n}\right)$
	and any coefficients $ \left( n_{j} \right) \subset \mathbb{C},$ we have
	\begin{align}\label{CH22}
	\int_{-\pi}^{\pi}\left( \int_{\mathbb{R}^{n}}\left| \sum_{j} n_{j}\left| \left(e^{-i t H} u_{j}\right) ( x )\right|^{2} \right|^{q} d x \right)^{\frac{p}{q}} d t  \leq C_{n, q}^{p}\left(\sum_{j}\left|n_{j}\right|^{\frac{2 q}{q+1}}\right)^{\frac{p(q+1)}{2 q}},
	\end{align}
	where $C_{n, q}$ is a universal constant which only depends on $n$ and $q$.
\end{thm}
\begin{proof}
Using the fact that the operator $e^{-i t H}$ is unitary, triangle inequality gives (\ref{CH22}) for the pair $(p,q)=(\infty,1)$. Equivalently, the operator $$W\in L_t^\infty L_x^2((-\pi, \pi)\times\mathbb{R}^n)\mapsto We^{-itH}(e^{-itH})^*\overline{W}\in\mathcal{G}^\infty $$ is bounded by Lemma \ref{CH2dual}.
Similarly, by Theorem \ref{CH2diag}, the operator $$W\in L_t^{n+2} L_x^{n+2}((-\pi, \pi)\times\mathbb{R}^n)\mapsto We^{-itH}(e^{-itH})^*\overline{W}\in\mathcal{G}^{n+2} $$ is bounded.
 Applying  the complex interpolation method \cite{ber}(chapter  4), the operator $$W\in L_t^{\frac{2q}{2-q}} L_x^{\frac{2p}{2-p}}((-\pi, \pi)\times\mathbb{R}^n)\mapsto We^{-itH}(e^{-itH})^*\overline{W}\in\mathcal{G}^\alpha $$ is bounded for $2\leq \frac{2q}{2-q}\leq n+2$ and $n+2\leq \frac{2p}{2-p}\leq\infty$.  Again applying Lemma \ref{CH2dual}, the inequality (\ref{CH22}) holds for the range $1\leq q\leq 1+\frac{2}{n}$. By Theorem \ref{CH2fr}, (\ref{CH22}) is valid when $1+\frac{2}{n}<q<\frac{n+1}{n-1}$.
\end{proof}
\section{Optimality of the Schatten exponent}\label{CH2sec5}
In this section, we show that the power $\frac{p(q + 1)}{2q}$ on the right hand side in (\ref{CH22})  is optimal. The inequality (\ref{CH22}) can also be written in terms of the operator
\begin{align}\label{CH2z}
\gamma_0 :=\sum_{j} n_{j}\left|u_{j}\right\rangle\left\langle u_{j}\right|
\end{align}
on $L^{2}\left(\mathbb{R}^{n}\right),$ where  the  Dirac's notation \(|u\rangle\langle v|\) stands  for the
rank-one operator $f \mapsto\langle v, f\rangle u$.  For such $\gamma_0$, let
$$\gamma(t) :=e^{-i t H} \gamma_0 e^{i t H}=\sum_{j} n_{j}\left|e^{-i t H} u_{j}\right\rangle\left\langle e^{-i t H} u_{j}\right|. $$
Then the density of the operator $\gamma(t)$ is given by
\begin{align}\label{CH2y}
\rho_{\gamma(t)} :=\sum_{j} n_{j}\left|e^{-i t H} u_{j}\right|^{2}.
\end{align}
 With these notations (\ref{CH22}) can be  rewritten  as
\begin{align}\label{CH27}
\left\|\rho_{\gamma(t)}\right\|_{L_{t}^{p}L_{x}^{q}\left((-\pi, \pi) \times \mathbb{R}^{n} \right)} \leq C_{n, q}\|\gamma_0\|_{\mathcal{G}^{\frac{2 q}{q+1}}},\end{align} where $\|\gamma_0\|_{\mathcal{G}^{\frac{2q}{q+1}}}=\left(\displaystyle\sum_j |n_j|^{\frac{2q}{q+1}}\right)^{\frac{q+1}{2q}}.$

\begin{prop}\label{CH26}
 (Optimality of the Schatten exponent). Assume that \(n, p, q \geq\)
1 satisfy \(\frac{2}{p}+\frac{n}{q}=n.\) Then we have
$$\sup _{\gamma_0 \in \mathcal{G}^{r}} \frac{\left\|\rho_{e^{-i tH} \gamma_0 e^{i t H}}\right\|_{L_{t}^{p} L_{x}^{q}\left((-\pi, \pi)\times \mathbb{R}^{n}\right)}}{\|\gamma_0\|_{\mathcal{G}^{r}}}=+\infty$$
for all \(r>\frac{2 q}{q+1}.\)
\end{prop}

\begin{proof}
Depending on the positive parameters $\beta, L$ and $\mu$, we construct the family of operators  $$\gamma_0=\frac{1}{(2 \pi)^{n}}\iint_{\mathbb{R}^{n} \times \mathbb{R}^{n}} e^{-\frac{x^{2}}{L^{2}}-\frac{\xi^{2}}{\mu}} \left |F_{x, \xi}\right\rangle\left\langle F_{x, \xi}\right| d x d \xi ,$$ where $F_{x, \xi}(z)=(2 \pi \beta)^{-\frac{n}{4}} e^{-\frac{(z-x)^{2}}{4 \beta}} e^{i \xi \cdot z}$. The functions  $F_{x, \xi}$ are normalized and satisfy
$$
\iint_{\mathbb{R}^{n} \times \mathbb{R}^{n}} \frac{d x d \xi}{(2 \pi)^{n}}\left|F_{x, \xi}\right\rangle\left\langle F_{x, \xi}\right|=1.
$$
By Mehler's formula we get
$$e^{i t H} F_{x, \xi}(z)= (-2 \pi i \sin2t )^{-\frac{n}{2}} (2 \pi \beta)^{-\frac{n}{4}}  \int_{\mathbb{R}^n}  e^{-\frac{i}{2} {\cot2t} \left(z^{2}+y^{2}\right)+\frac{i}{\sin (2 t)} z \cdot y} e^{-\frac{(y-x)^{2}}{4 \beta}} e^{i \xi \cdot y} dy.$$
Therefore
$$\left| e^{i t H} F_{x, \xi}(z)\right| = \left (  \frac{2\beta}{\pi (4\beta ^2 \cos^2 2t +\sin^2 2t)} \right )^{\frac{n}{4}} e^{-\frac{\beta(z-x\cos 2t +\xi \sin 2t)^2}{4\beta ^2 \cos^2 2t +\sin^2 2t}}$$ and
\begin{align*} \rho_{\gamma(t)}(z) & :=\rho_{e^{i tH} \gamma_0 e^{-i t H}(z)} \\
&=\iint_{\mathbb{R}^{n} \times \mathbb{R}^{n}} \frac{d x d \xi}{(2 \pi)^{d}} e^{-\frac{x^{2}}{L^{2}}-\frac{\xi^{2}}{\mu}}  \left|e^{i t H} F_{x, \xi}(z)\right|^{2} \\
&=\left( \frac{ 2 \pi\beta \mu L^{2}}{\left(4\beta ^2+2\beta L^2\right) \cos^2 2t +(1+2\mu \beta ) \sin^2 2t}\right)^{\frac{n}{2}} e^{\frac{-2\beta z^2}{\left(4\beta ^2+2\beta L^2\right ) \cos^2 2t +\left (1+2\mu \beta \right ) \sin^2 2t } }.
\end{align*}
So
\begin{eqnarray*}
\|\rho_{\gamma(t)}\|_{L_x^q(\mathbb{R}^n)}^q &=&\left( \frac{\pi}{q}\right)^{\frac{n}{2}}\left(\mu L^2\right)^{\frac{nq}{2}} \left( \frac{\beta }{\left(4\beta ^2+2\beta L^2\right) \cos^2 2t +\left (1+2\mu \beta \right ) \sin^2 2t }\right)^{\frac{n(q-1)}{2}}.
\end{eqnarray*}
 Using the fact that $n(q-1)p=2q$, we have
\begin{align*}
&\|\rho_{\gamma(t)}\|_{L^p_t L^q_x((-\pi, \pi)\times \mathbb{R}^n)}^p\\&= \left( \frac{\pi}{q}\right)^{\frac{np}{2q}}\left( \mu L^2\right)^{\frac{np}{2}} \int_{[-\pi, \pi ]}   \frac{\beta }{\left(4\beta ^2+2\beta L^2\right ) \cos^2 2t +\left (1+2\mu \beta \right ) \sin^2 2t } dt\\
&= \sqrt{2}\pi  \left( \frac{\pi}{q}\right)^{\frac{np}{2q}}\left( \mu L^2\right)^{\frac{np}{2}} \frac{\beta }{\sqrt{2\beta ^2+\beta L^2}\sqrt{1+2\mu \beta}}.
\end{align*}
 Thus
 \begin{align*}
&\|\rho_{\gamma(t)}\|_{L^p_t L^q_x((-\pi, \pi)\times \mathbb{R}^n)}\\
&=A_{n,p}\left( \mu L^2\right)^{\frac{n}{2}}   (L^2)^{-\frac{1}{2p}}   \mu^{-\frac{1}{2p}}  \frac{1 }{\left (\frac{2\beta}{L^2} + 1\right )^{\frac{1}{2p}}\left (\frac{1}{\mu \beta }+2 \right )^{\frac{1}{2p}}}\\
&=A_{n,p}\left( \mu L^2\right)^{\frac{n}{2}-\frac{1}{2p}}  \frac{1 }{\left (\frac{2\beta}{L^2} + 1\right )^{\frac{1}{2p}}\left (\frac{1}{\mu \beta }+2 \right )^{\frac{1}{2p}}}.
 \end{align*}

 Using the fact that $ \frac{n}{4}\left(1+\frac{1}{q} \right )=\frac{n}{2}-\frac{1}{2p}$ and choosing \(1 / \mu \ll \beta \ll L^{2},\) we obtain
\begin{align*}
&\|\rho_{\gamma(t)}\|_{L^p_t L^q_x((-\pi, \pi)\times \mathbb{R}^n)}
\approx A_{n,p}\left( \mu L^2\right)^{\frac{n}{4}\left(1+\frac{1}{q} \right )} 2^{-\frac{1}{2p}}\approx A_{n,p}2^{-\frac{1}{2p}} \left( \mu L^2\right)^{\frac{n}{2}\left(\frac{1+q}{2q} \right )} \approx A_{n,p}2^{-\frac{1}{2p}} N^{\frac{1+q}{2q}},
\end{align*} where
\begin{align*}
N&=\int_{\mathbb{R}^{n}} \gamma_0(z, z) d z\\
&=\iiint_{\mathbb{R}^{n} \times \mathbb{R}^{n} \times \mathbb{R}^{n}} \frac{d x d \xi}{(2 \pi)^{n}} e^{-\frac{x^{2}}{^{2}}-\frac{\xi^{2}}{\mu}} \left| F_{x\xi}(z)\right|^2 dz\\
&=\iint_{\mathbb{R}^{n} \times \mathbb{R}^{n}} \frac{d x d \xi}{(2 \pi)^{n}} e^{-\frac{x^{2}}{^{2}}-\frac{\xi^{2}}{\mu}} \\
&=A_{n} L^{n} \mu^{\frac{n}{2}}.
\end{align*}
An application of  Berezin-Lieb inequality gives that
$$\operatorname{Tr} \gamma_0^{r} \leq \iint_{\mathbb{R}^{n} \times \mathbb{R}^{n}} \frac{d x d \xi}{(2 \pi)^{d}} e^{-\frac{rx^{2}}{^{2}}-\frac{r\xi^{2}}{\mu}}=r^{-n} N,$$ where  $r\geq 1$ and $N=\frac{(\mu L^2)^{\frac{n}{2}}}{2^n}.$
Therefore $$\frac{\left\|\rho_{e^{-i tH} \gamma_0 e^{i t H}}\right\|_{L^p_t L^q_x((-\pi, \pi)\times \mathbb{R}^n)}}{\|\gamma_0\|_{\mathcal{G}^{r}}}\geq \frac{A_{n,p}2^{-\frac{1}{2p}}}{r^{-\frac{n}{r}}}N^{\left(\frac{1+q}{2q}-\frac{1}{r}\right)}.$$
\end{proof}

Using the similar idea as in Theorem 14 of \cite{frank1}, the following global well-posedness for the Hermite-Hartree equation in Schatten spaces can be obtained as an application of Theorem \ref{CH2STH}.
\begin{thm}(Global well-posedness for the Hermite-Hartree equation)
 Let $w\in L^{q^\prime}(\mathbb{R}^n)$. Under the assumptions of Theorem \ref{CH2STH},
	  for any $\gamma_0  \in  {\mathcal{G}^{\frac{2 q}{(q+1)}}}$, there exists a unique
	$\gamma \in  C_t^0([0, T], {\mathcal{G}^{\frac{2 q}{(q+1)}}} )$
	satisfying
	$\rho_{\gamma} \in  L_{t}^p L_x^q \left([0, T]\times \mathbb{R}^n\right)$
	 and
	$$i\partial_t \gamma = [H+w\star\rho_\gamma, \gamma], \quad \gamma|_{t=0 }= \gamma_0.$$
	\end{thm}
\section*{Acknowledgments}
The first author thanks  Indian Institute of Technology Guwahati for the support provided during the period of this work.

\end{document}